\documentclass[reqno,11pt,letterpaper]{amsart}

 \usepackage{comment}
 \usepackage{amsmath,amssymb,amsthm,graphicx,mathrsfs,url}
\usepackage[usenames,dvipsnames]{color}
\usepackage[colorlinks=true,linkcolor=ForestGreen,citecolor=ForestGreen]{hyperref}
\usepackage{amsxtra}
\usepackage{epstopdf}
\usepackage[dvipsnames]{xcolor}
\usepackage{graphicx}

\def\?[#1]{\textbf{[#1]}\marginpar{\Large{\textbf{??}}}}
\def\smallsection#1{\smallskip\noindent\textbf{#1}.}
\numberwithin{equation}{section}

\usepackage{MnSymbol}
\usepackage{bbm}
\newcommand{\norm}[1]{\left\| #1 \right\|}

\allowdisplaybreaks

\numberwithin{equation}{section}
\newcounter{exercise}

\newtheorem{theorem}{Theorem}
\newtheorem{proposition}[theorem]{Proposition}
\newtheorem{lemma}[theorem]{Lemma}
\newtheorem{corollary}[theorem]{Corollary}

\theoremstyle{definition}

\theoremstyle{remark}
\newtheorem{remark}[theorem]{Remark}

\newcommand\R{{\mathbb R}}

\newcommand{\dd}{{\, \mathrm d}}

\let\oldmarginpar\marginpar
\renewcommand\marginpar[1]{\-\oldmarginpar[\raggedleft\footnotesize #1]%
{\raggedright\footnotesize #1}}

\def\smallsection#1{\smallskip\noindent\textbf{#1}.}
\let\epsilon=\varepsilon % sorry Knuth

\title[Global smooth solutions for triangular reaction-cross diffusion systems]{Global smooth solutions for triangular reaction-cross diffusion systems}

\author{Jessica Guerand}
\address[Jessica Guerand]{IMAG, Université de Montpellier, 499-554 Rue du Truel, 34090 Montpellier, France}
\email{jessica.guerand@umontpellier.fr}
\author{Angeliki Menegaki}
\address[Angeliki Menegaki]{Institut des hautes études Scientifiques, 35 Rte de Chartres, 91440, Bures-sur-Yvette, France}
\email{menegaki@ihes.fr}
\author{Ariane Trescases}
\address[Ariane Trescases]{IMT; UMR5219, Universit\'e de Toulouse; CNRS, F-31400 Toulouse, France}
\email{ariane.trescases@math.univ-toulouse.fr}

\begin{document}

\maketitle

\sloppy
\begin{abstract}
For a class of reaction cross-diffusion systems of two equations with a cross-diffusion term in the first equation and with self-diffusion terms, we prove that the unique local smooth solution given by Amann theorem is actually global. This class of systems arises in Population dynamics, and extends the triangular Shigesada-Kawasaki-Teramoto system when general power-laws growth are considered in the reaction and diffusion rates.\end{abstract}

\setcounter{secnumdepth}{7}
\setcounter{tocdepth}{1}
\tableofcontents 
 
\section{Introduction} 

 \subsection{The system}
The purpose of this article is to study the global existence of a smooth solution of the following system, 
 
\begin{equation}  \label{eq:the system} \begin{array}{rl}
\partial_t u - \Delta \Big( ( d_u+d_{\alpha} \langle u \rangle ^{\alpha} + d_{\beta} \langle v \rangle ^{\beta} ) u \Big) = u(r_u-r_a \langle u \rangle^a-r_b \langle v \rangle^b),&  \text{in}\ [0,\infty) \times \Omega, \\   
\partial_t v - \Delta \big(( d_v +d_{\gamma}\langle v \rangle ^{\gamma}) v \big) = v(r_v-r_c \langle v \rangle^c-r_d \langle u \rangle^d), & \text{in}\  [0,\infty) \times \Omega, \\
\nabla u  \cdot n =  \nabla v  \cdot n =0, & \text{in}\  [0,\infty) \times \partial \Omega, \\
u(0,\cdot)= u_0, \quad v(0,\cdot)= v_0, & \text{in}\ \Omega,
\end{array}
\end{equation}
with the bracket $\langle \cdot \rangle $ defined by
\begin{equation}
\langle z \rangle  := \left( 1+z^2 \right)^{1/2}, \qquad z \ge 0,
\end{equation}
and where $u=u(t,x) \geq 0$ and $v=v(t,x) \geq0$ are the unknowns, $\Omega$ is a smooth open bounded domain of $\mathbb{R}^m$ for $m\geq 2$ and $n(x)$ is the outward normal vector at point $x \in \partial \Omega$. The functions $u_0, v_0 : \Omega \rightarrow \mathbb{R}$ are nonnegative initial data belonging to $C^{2+\nu}$ for some $\nu>0$.
The constant parameters are defined in the set $D$,
\small{\begin{equation}
\label{def cst}
D:= \{d_u, d_v, d_{\alpha}, d_{\beta}, d_{\gamma}, r_{u}, r_{v}, r_{a}, r_{b}, r_{c}, r_{d}, a, b, c, d, \alpha, \beta, \gamma \} \in \left((0,\infty)\right)^{17}\times [0,\infty),
\end{equation}}\normalsize
and are assumed to satisfy
\begin{equation}
\label{assump cst}
d < \frac{1}{2} \max\{ 1+a, 2+ \alpha, a+\alpha, 2\alpha+2/m\} \quad \text{and} \quad d < \max\left(\frac{8\alpha(m+1)}{(m^2-4)_+},\frac{4\alpha +2a}{m+2}\right),
\end{equation}
where we recall that $m$ is the space dimension.
%Note in particular that we assume the presence of self-diffusion in the first equation ($d_\alpha>0$ and $\alpha>0$), but not necessarily in the second equation ($\gamma\ge0$). However, in the absence of self-diffusion in the second equation ($\gamma=0$), our result is not optimal.
This system is a prototypical extension of the triangular Shigasada-Kawasaki-Teramoto system arising in Population dynamics, \cite{SKT}:
 \begin{equation}  \label{sys:SKTq} \begin{array}{rl}
\partial_t u - \Delta \Big( ( d_u+d_{\alpha} u + d_{\beta}v) u \Big) = u(r_u-r_a u-r_bv),&  \text{in}\ [0,\infty) \times \Omega, \\   
\partial_t v - \Delta \big( (d_v +d_{\gamma}v) v \big) = v(r_v-r_c v-r_d u), & \text{in}\  [0,\infty) \times \Omega ,%\\
%\nabla u  \cdot n(x) =  \nabla v  \cdot n(x) =0\ & \text{for}\ x \in \partial \Omega \\
%u(0,.)= u_0, \quad v(0,.)= v_0 & \text{for } x \in \Omega.
\end{array} 
\end{equation}
when the linear terms in the reaction and diffusion rates are replaced with functionals with a power law growth at infinity. System \eqref{sys:SKTq} and its extension System \eqref{eq:the system} model the evolution of the space densities of the populations of two living species interacting through their movement. The terms $d_\alpha \langle u \rangle ^{\alpha} u$ and $d_\gamma \langle v \rangle ^{\gamma} v$ are the self-diffusion terms, and $d_{\beta}\langle v \rangle ^{\beta} u$ is the cross-diffusion term. In the diffusion rates, we use the bracket $\langle \cdot \rangle$ taken at some powers $\alpha$, $\beta$ or $\gamma$ to guarantee a smooth behaviour close to zero even for powers $\alpha$, $\beta$ or $\gamma$ less than one. When the power $\alpha$, $\beta$ or $\gamma$ respectively is greater or equal to one, one could replace the term $\langle u \rangle^\alpha$, $\langle v \rangle^\beta$ or $\langle v \rangle^\gamma$ respectively, by the simpler power law term $u^\alpha$, $v^\beta$ or $v^\gamma$ respectively, and our results would still apply. The same applies for the reaction rates $\langle u \rangle^a$, $\langle v \rangle^b$, $\langle v \rangle^c$ and $\langle u \rangle^d$.

% In the case where $\alpha, \beta, \gamma \geq 1$, we could consider $u^{\alpha} $, $v^{\beta}$ and  $v^{\gamma}$ instead of $\langle u \rangle ^{\alpha}$, $\langle v \rangle ^{\beta}$ and $\langle v \rangle ^{\gamma}$ in \eqref{eq:the system}.

By considering System \eqref{eq:the system} instead of the original system \eqref{sys:SKTq} we follow the idea of replacing linear interaction rates by more general nonlinear terms. This idea goes back to the studies of \cite{GJ72,GA73} where some non-linear functions were proved to be more appropriate than linear ones to model the competitive interaction between two species (in their study, two species of drosophila), modelled in that case with an ordinary differential system. We use power laws for simplicity, though a wide range of functions could be considered. Such generalized (triangular) Shigasada-Kawasaki-Teramoto systems have been studied for example in \cite{PozioTesei,Wang2005,Yamada,DT15,AT16}.

 \subsection{Main result}
Our main result is the following. 
 \begin{theorem}[Global smooth solutions]
 \label{th: main}
 Let $\Omega$ be a bounded open subset of $\R^m$ such that $\partial \Omega$ is smooth. Let the coefficients of \eqref{eq:the system} satisfy \eqref{def cst}--\eqref{assump cst} and  $u_0, v_0$ be nonnegative functions in $C^{2+\nu}(\overline{\Omega})$ for some $\nu>0$ satisfying the Neumann boundary condition $$\nabla u_0 \cdot n =\nabla v_0 \cdot n=0 \quad \mbox{ on } \partial \Omega.$$ Then the system possesses a unique, nonnegative, global solution $(u,v)$ such that 
 $$u,v \in C^{\frac{2+\nu}{2},2+\nu}([0,\infty)\times \overline{\Omega})\cap C^{\infty}((0,\infty)\times\overline{\Omega}).$$
 \end{theorem}
% \jessica{I guess there are many ways to write this statement  depending on the conditions of $u_0$, $v_0$}
 
To prove this theorem, we rely on Amann's results which gives the local existence of a unique smooth solution together with an extension criteria (for a large classs of system including \eqref{eq:the system}, see \cite{Am90,Am89,Am93}). More precisely, for System \eqref{eq:the system}, Amann's result reads as follows. 
\begin{theorem}[Local smooth solutions \cite{Am90,Am89,Am93}]
\label{amann}\label{thm:local}
 Let $\Omega$ be a bounded open subset of $\R^m$ such that $\partial \Omega$ is smooth. Let the coefficients of \eqref{eq:the system} satisfy \eqref{def cst} and  $u_0, v_0$ be nonnegative functions in $W^{1,p_0}(\Omega)$ with $p_0>m$. Then there exists a maximal time $t_{\mathrm{max}} \in \left(0,\infty\right]$ such that the system \eqref{eq:the system} has a unique nonnegative solution $(u,v)$ in $ \left(0,t_{\mathrm{max}}\right)\times \Omega$ such that 
 $$u,v \in C([0,t_{\mathrm{max}}),W^{1,p_0}(\Omega))\cap C^{\infty}((0,t_{\mathrm{max}})\times\overline{\Omega}),$$
Moreover, if $t_\mathrm{max}<\infty$, then 
$$\lim\limits_{t\rightarrow t^{-}_{\mathrm{max}}} \left[\norm{u(t,\cdot)}_{W^{1,p_0}(\Omega)} + \norm{v(t,\cdot)}_{W^{1,p_0}(\Omega)}  \right]= \infty.$$ 
\end{theorem}
 
 \subsection{State of the art}
The local existence of smooth solutions together with an extension criteria for a wide class of parabolic systems including \eqref{eq:the system} was established by Amann in \cite{Am90,Am89,Am93}, as stated in Theorem \ref{amann}.
 
For global existence, many works focus on the original (quadratic) triangular Shigesada-Kawasaki-Teramoto system \eqref{sys:SKTq}. %Most works rely on the application of Amann's result, and the main goal is to extend the local solutions of Amann by obtaining bounds in appropriate Sobolev spaces.
When there is self-diffusion in the first equation, that is when $d_\alpha>0$, the existence of global smooth solutions was first established in dimensions $m=1$ and $m=2$, \cite{LouNiWu,Yagi}, then in dimension $m\le5$ in \cite{ChoiLuiYamada04,LeNguyenNguyen}, then $m \le 9$ in \cite{VTP08}. Finally, Hoang \emph{et al.} recently extended these results to any dimension $m\ge1$ in \cite{HNP2}. Their method rely on De Giorgi techniques and on the proof of a Sobolev regularity result for nonlinear parabolic scalar equations with self-diffusion. We also mention the interesting work of \cite{TaoWinkler_stab} where the authors show that for the case $d_\alpha>0$, in convex domains of dimension $m \le 9$, the global solutions are uniformly bounded and the long-time behaviour is studied.
The situation without self-diffusion in the first equation ($d_\alpha=0$) is more delicate, and global smooth solutions have been obtained only in dimension $m=2$ in \cite{LouNiWu}, with the exception of the case with no self diffusion at all ($d_\alpha = d_\gamma =0$) which is treated in any dimension $m\ge 1$ in \cite{ChoiLuiYamada04} and \cite{DT15}. However, global weak solutions exist in any dimension $m$ even for the case $d_\alpha=0$ and $d_\gamma>0$: see \cite{AT16}.
% We have not mentionned the numerous works where cross-diffusion is assumed weak in some sense; nor the non triangular case.

Let us now present works on more generic forms of triangular cross-diffusion systems allowing for power law growth in the spirit of \eqref{eq:the system}. In absence of self-diffusion terms, that is $\alpha=\gamma=0$, global smooth solutions were obtained in \cite{PozioTesei} with assumptions on the reaction part that amount to assume a fast decay in $u$ of the reaction terms and in \cite{Yamada} when $a>d$. Global weak solutions were furthermore obtained in \cite{DT15} in the same case $a>d$ and when $a\le 1$ and $d\le 2$ (without the brackets $\langle \cdot \rangle$). Allowing the presence of self-diffusion, global existence has been obtained in \cite{Wang2005} in the case $\alpha>0$, $\gamma=0$ under a (dimension-dependent) condition of smallness of the parameter $d$ w.r.t. the parameter $a$. Then, global weak solutions were obtained in \cite{AT16} when the powers satisfy: $(\alpha>0, d<2+\alpha, a <1+ \alpha)$ or $(\alpha=0, d \le 2, a \le 1)$, with no condition on $\gamma$ other that nonnegativity (and again, without the need of the brackets $\langle \cdot \rangle$). We also mention here the recent paper \cite{Le21} where regularity criteria leading to uniqueness of a specific weak-strong notion of solutions are provided for a class of cross-diffusion systems. 

In the general case where $\gamma$ can be positive, smooth solutions were not studied, which is the purpose of the present work. The case $\gamma>0$ is the most delicate case, as one can not rely on the properties of the heat operator: see Remark \ref{rk:gamma0}. In particular, for the original SKT system \eqref{sys:SKTq}, we observe from the literature that the presence of self-diffusion in the first equation $\alpha>0$ is crucial to obtain regularity. Here, we want to investigate how much self-diffusion, quantified by the parameter $\alpha$, is actually enough to obtain smooth solutions for the more general system \eqref{eq:the system}. A first answer is given by the assumption \eqref{assump cst}. However, we give in the following subsection some possible relaxations of these assumptions.

% \jessica{Add in particular the comparison with existing results}
 
 \subsection{Strategy of the proof and possible extensions}\label{sec:strategy}
 Our strategy of proof relies on the following main ingredients: first we apply Amann's theorem to obtain local smooth solutions on $[0,t_\mathrm{max})$, and the objective is to extend them globally. Let $T \in (0,\infty)$ with $T \leq t_{\mathrm{max}}$. Our objective is to prove that the local solution satisfy
\begin{equation}
\label{condition on the grad}
\sup_{t\in (0,T )}\norm{u(t,\cdot)}_{W^{1,p_0}(\Omega)} + \norm{v(t,\cdot)}_{W^{1,p_0}(\Omega)} < \infty,
\end{equation}
so that necessarily by Theorem \ref{amann},  the unique local solution in Theorem \ref{amann} is global, which gives Theorem \ref{th: main}.
 
Our first basic estimates are the maximum principle for $v$ and a combination of a duality estimate and an energy-type estimate coming from the logistic-type growth for $u$. Then, we iterate some higher-order energy-type estimates for $u$ and classical parabolic regularity estimates for $v$ in Sobolev spaces in order to obtain that both $u$ and $\nabla v$ are in $L^p((0,T)\times\Omega)$ for all $p>1$. We then obtain that $u$ is bounded by defining a class of energy estimate that allows us to use De Giorgi method. And finally we conclude using the classical H\"older and Schauder estimates. 

We list in the following remarks some possible extensions of this result and perspectives.
\begin{remark}[The case $\alpha=0$]
When $\alpha>0$, the self-diffusion in the first equation allows to prove more regularity for $u$. This appears in our assumption \eqref{assump cst}, and can also be seen directly for example in the duality estimate which gives a bound for $u$ in $L^{2+\alpha}$. The case where $\alpha=0$ is an unfavorable case. It is actually covered by our method when the reaction parameters $a$ and $d$ satisfy $d (m+2)<2a$.
\end{remark} 
\begin{remark}[The case $\gamma=0$]\label{rk:gamma0}
The case where there is no self-diffusion in the second equation, that is $\gamma=0$, is a very favorable case. In this case, the properties of the heat operator allow to obtain regularity for the first and second space derivatives of $v$ without any dimension-dependent assumption, and this even in the unfavorable case where there is no self-diffusion in the first equation. For example, the results hold under the condition $\gamma=0$, $\alpha=0$, $d<a$, see \cite{Yamada} or \cite{DT15}.
\end{remark}
\begin{remark}[Dimension-dependent assumptions]
We expect that our result holds for larger range for $d$ rather than $d < \max\left(\frac{8\alpha(m+1)}{(m^2-4)_+},\frac{4\alpha +2a}{m+2}\right)$. We use this assumption in the proof of Lemma \ref{u grad v any Lp space} in order to obtain the continuity of $v$ (thanks to dimension-dependent embeddings), which allows to apply classical results of parabolic regularity in Sobolev spaces: see the result due to Ladyzenskaya \emph{et al.} stated later in Lemma \ref{W12 estimates}. In fact, for the original system \eqref{sys:SKTq} we recover the results in \cite{VTP08} with the same restriction on the dimension, see also \cite{TaoWinkler_stab}.
However, some recent results for the original system \eqref{sys:SKTq} due to Hoang \emph{et al.}, \cite{HNP2}, suggest that we can bypass this classical tool and obtain directly estimates on the gradient of $v$ in any $L^p$ space, conditionally to the $L^p$ regularity of the right hand side, without requiring the continuity of $v$. The result of Hoang \emph{et al.} could be extended to system \eqref{eq:the system} with a generic $\gamma>0$, and as a consequence we could remove the space-dependent assumption. This is the purpose of an ongoing work.
\end{remark}
\begin{remark}[Weak solutions] Using the same techniques we could also prove that the weak solutions obtained in \cite{AT16} with initial conditions in $W^{2,\infty}$, are actually smooth. In that case one can deduce that the two solutions in \cite{Am93} and in \cite{AT16} must coincide.
\end{remark}
 
\subsection{Notations} 
We end this section by giving some notations and definitions.

Recall that $T\in (0,\infty)$ is defined in Sec. \ref{sec:strategy}, with $T\le t_\mathrm{max}$, and $(u,v)$ are nonnegative solutions of \eqref{eq:the system} on $\left[0,T\right)\times \Omega$ given by Theorem \ref{thm:local}.

We define the domain $\Omega_T=(0,T)\times \Omega$. Throughout this paper, $g\in W^{2}_p(\Omega)$ means $g, \partial_{x_i} g,\partial_{x_i x_j} g$ are in $L^p(\Omega)$, and  $g\in W^{1,2}_p (\Omega_T)$ means $g, \partial_{x_i} g,\partial_{x_i x_j} g $ and $\partial_t g$ are in $L^{p}(\Omega_T)$ for $i,j=1,\dots,m$. The space $W^{1,2}_p(\Omega_T)$ is endowed with the norm
$$ \| g\|_{W^{1,2}_p(\Omega_T)}=\| g\|_{L^p(\Omega_T)}+\| \partial_t g\|_{L^p(\Omega_T)}+\| \nabla g\|_{L^p(\Omega_T)}+ \sum\limits_{i,j=1}^{m} \| \partial_{x_i x_j} g\|_{L^p(\Omega_T)}.$$
Here $\nabla g= (\partial_{x_1}g, \dots, \partial_{x_m} g)$. 
We define also the positive part of a function $g_+=\max (g,0)$. 
We occasionally write $f \lesssim g$ in order to say that $f \le Cg$ for some constant $C$ 
which depends on time $T$, domain $\Omega$ and the parameters $D$. %Let us give the definition of weak solution that we use in this paper. 

 \section{Gain of integrability of $\nabla_x v$ and $u$}
 
In this section we prove that $u$ and $\nabla_x v$ are in any $L^p$ space with $p\in (0,\infty)$.  To get this result, we use an iteration process to transfer of regularity between $u$ and $\nabla v$. We first give some basic estimates to initialize the iteration process. 
Then we give some conditional properties of transfer of regularity between $u$ and $\nabla_x v$ and we conclude by iteration. 

\subsection{Basic estimates}

We start with the following basic estimates.
\begin{lemma}[Basic estimates]
\label{lem:basic}
The solution $(u,v)$ of \eqref{eq:the system} satisfies the following bounds
\begin{align} \label{basic}
\sup_{t\in(0,T)} \int_\Omega u(t,x) \dd x \le C, \qquad \iint_{\Omega_T} u^{1+a}(t,x) \dd x \dd t \le C_1(T),\\
\sup_{t\in(0,T)} \|v(t,\cdot)\|_{L^\infty(\Omega)} \le C,
\end{align} 
where $C$ depends only on the initial data, the domain $\Omega$ and the parameters $D$, while $C_1(T)$ depends linearly on time.
\end{lemma}
\begin{proof}
Let us integrate the first equation of \eqref{eq:the system} on $\Omega_t$ for $t\in \left(0,T\right]$
\begin{equation}
\label{integr}
\int_\Omega u (t,x) \dd x = \int_\Omega u (0,x) \dd x + \iint_{\Omega_T} u\, (d_u - d_a \langle u\rangle^a-d_b \langle v\rangle^b) \dd x \dd s.
\end{equation}
Using the nonnegativity of $u$ and $v$ and the fact that $d_u z - d_a z\langle z\rangle^a + \frac{d_a}2 z^{a+1}  \le C$ for all $z\ge 0$, we get the estimate using twice \eqref{integr}, once for time $t$ and once for time $T$,
\begin{equation}
\int_\Omega u (t,x) \dd x + \iint_{\Omega_T} \frac{d_a}2 u^{1+a} \dd x \dd t \le \int_\Omega u (0,x) \dd x + C |\Omega| T,
\end{equation}
which concludes the two first estimates. The last estimate is a direct consequence of the maximum principle.
\end{proof}

\begin{lemma}[Interpolation Inequality] 
\label{lem:interpolation}
Let $w$ be the function $w=(d_v+d_{\gamma}\langle v \rangle ^\gamma)v$ and $p \geq 1$. It holds that
$$
 \|w \|_{L^{2p}((0,T),W^{1,2p}(\Omega))}^{2p} 
 \leq  C \|w \|^{p}_{L^{p}((0,T),W^{2,p}(\Omega))}.$$
%$$L^{p}((0,T),W^{2,p}(\Omega))\cap L^{\infty}((0,T),L^{\infty}(\Omega)) \subset L^{2p}((0,T),W^{1,2p}(\Omega)).$$
\end{lemma}
\begin{proof}
To prove this we use $W^{2,p}(\Omega)\cap L^{\infty} (\Omega) \subset W^{1,2p}(\Omega)$ of \cite[Theorem 6.4.5 p.153]{BL12} or \cite[Theorem 1]{BM19}  which gives
$$ \|w \|_{W^{1,2p}(\Omega)} \leq C \|w \|^{\theta}_{L^{\infty} (\Omega)} \|w \|^{1-\theta}_{W^{2,p}(\Omega)},$$
so  using $\theta=\frac{1}{2}$  we have  
\begin{align*}
\displaystyle \int_{0}^{T} \|w \|_{W^{1,2p}(\Omega)}^{2p} \mathrm{d}t &\leq C \displaystyle \int_{0}^{T}  \|w \|^{2p\theta}_{L^{\infty} (\Omega)} \|w \|^{2p(1-\theta)}_{W^{2,p}(\Omega)} \mathrm{d}t \\
& \leq C \|w \|^{p}_{L^{\infty} (\Omega_T)} \displaystyle  \int_{0}^{T}  \|w \|^{p}_{W^{2,p}(\Omega)} \mathrm{d}t \\
& \leq C \|w \|^{p}_{L^{p}((0,T),W^{2,p}(\Omega))},
\end{align*}
where we have used Lemma \ref{lem:basic}  to bound the $L^{\infty}$-norm in the right-hand side.\end{proof}

The next basic estimate is a consequence of the following lemma from \cite{DLMT}, Lemma 2.11, which is obtained by duality techniques.
\begin{lemma}[Duality estimate]\label{lem:dual}
Let $M:[0,T]\times\Omega\rightarrow\R_+$ be a positive continuous function lower bounded by a positive constant and let $r_u>0$. Any smooth nonnegative solution of the differential inequality 
\begin{align*}
\partial_t u-\Delta (M u)\leq r_u u \text{ on }\Omega,\\
\partial_n (M u)=0,\text{ on }\partial\Omega,
\end{align*}
satisfies the following bound
\begin{align*}
\iint_{\Omega_T} M u^2\leq\exp(2 r_u T)\times \left( C_\Omega^2\, \|u(0,\cdot)\|_{H^{-1}_m(\Omega)}^2+\overline{ u(0,\cdot)}^2\iint_{\Omega_T} M\right),
\end{align*}
with $\overline{ u(0,\cdot)}$ the mean value of $u(0,\cdot)$ on $\Omega$ and $C_\Omega$ the Poincar\' e-Wirtinger constant.
\end{lemma}

As a consequence, we get the following duality estimate.
\begin{corollary}[Duality estimate]
\label{dual}
Let $(u,v)$ be the solution of \eqref{eq:the system}. Then $u$ satisfies the following bound,
\begin{equation} \label{dual}
\iint_{\Omega_T} u^{2+\alpha}(t,x) \dd x \dd t \le C,
\end{equation} 
where $C$ depends only on the initial data, the domain $\Omega$, the time $T$ and the parameters $D$.
\end{corollary}
\begin{proof}
Applying Lemma \ref{lem:dual} with $M:=d_u + d_\alpha \langle u \rangle ^\alpha + d_\beta \langle v \rangle ^\beta$ yields the estimate
\begin{equation}
\iint_{\Omega_T} \left( u^2 + u^{2+\alpha} + u^2v^\beta \right) (t,x) \dd x \dd t \le C \left( 1 + \iint_{\Omega_T} \left( u^{\alpha} + v^\beta \right) (t,x) \dd x \dd t \right)
\end{equation}
where we used the fact that $u^\alpha \le \langle u\rangle^\alpha \leq C (1+u^\alpha)$. Using now the
$L^\infty$-bound of $v$ from Lemma \ref{lem:basic} and the fact that $C z^{\alpha} \le C' + \frac{1}2 z^{2+\alpha}$ for all $z\ge 0$, we conclude the proof.
\end{proof}

\subsection{Transfer of integrability from $\nabla_x v$ to $u$}  
 
\begin{lemma}[Power energy estimate for $u$]
\label{power energ u}
Let $\rho\geq \max(1,\alpha)$ and $p>2$. Considering the system of cross-diffusion equations given by \eqref{eq:the system} with $\alpha>0$ and assuming that $\nabla v \in L^p(\Omega_{T})$ and $u\in L^{\rho}(\Omega_T)\cap L^{\frac{p}{p-2}(\rho-\alpha)}(\Omega_T)$, the following energy estimate holds true for all $s <t \in \left[0,T\right]$,
\begin{align} \label{power energy est for u} &\int_{\Omega} u^{\rho}(t,x) \dd x  +  \iint_{(s,t)\times\Omega}\vert \nabla u^{\frac{\rho}{2}+\frac{\alpha}{2}} \vert^2(\tau,x) \dd x \dd \tau  +  \iint_{(s,t)\times \Omega}  u^{a+\rho}(\tau,x) \dd x \dd \tau
  \\ \lesssim&  \int_{\Omega} u^{\rho}(s,x) \dd x + \left( \iint_{(s,t)\times\Omega} |\nabla v|^{p} (\tau,x) \dd x \dd\tau  \right)^{\frac{2}{p}}   \left( \iint_{(s,t)\times\Omega} u^{\frac{p}{p-2}(\rho-\alpha)} (\tau,x) \dd x \dd \tau  \right)^{\frac{p-2}{p}} +1 \notag,
\end{align} 
where the constant depends on the initial data, $D$, $\Omega$, $T$ and $\rho$.  
\end{lemma}

\begin{proof}
We multiply the equation for $u$ by $u^{\rho-1}$ and integrate on $(s,t)\times \Omega$
  \begin{align*}
&\iint_{(s,t)\times \Omega} \frac{\partial_\tau u^{\rho}}{\rho} \dd x\dd \tau
  + 2\frac{\rho-1}{\rho}  \iint_{(s,t)\times \Omega} \beta d_{\beta} \frac{v}{\langle v \rangle ^{2-\beta}}u^{\frac{\rho}{2}} \nabla v\cdot \nabla u^{\frac{\rho}{2}} \dd x\dd \tau \\
&+ 4\frac{\rho-1}{\rho^2} \iint_{(s,t)\times \Omega} \left( \left(d_u+d_{\alpha}\left(1+\alpha\frac{ u^2}{\langle u \rangle ^2}\right)\langle u \rangle ^{\alpha} + d_{\beta}\langle v \rangle ^{\beta} \right)  \vert \nabla u^{\frac{\rho}{2}}\vert^2  \right) \dd x\dd \tau  \notag \\
&= \iint_{(s,t)\times \Omega} u^{\rho} (r_u-r_a\langle u\rangle^a-r_b\langle v\rangle^b)\dd x\dd \tau.
\end{align*}
Using the fact that $z^\rho(r_u-\frac{r_a}{2}z^a)\le C$ for all $z\ge0$ we get
  \begin{equation}\label{simpl}
  \begin{split} 
\int_{\Omega} u^{\rho}(t,x) \dd x&+ C_1 \iint_{(s,t)\times \Omega} \vert \nabla u^{\frac{\rho}{2}+\frac{\alpha}{2}}\vert^2   \dd x\dd \tau  + 
C_2 \iint_{(s,t)\times \Omega} \frac{v}{\langle v \rangle ^{2-\beta}}u^{\frac{\rho}{2}-\frac{\alpha}{2}} \nabla v\cdot \nabla u^{\frac{\rho}{2}+\frac{\alpha}{2}} \dd x\dd \tau \\ 
&+
\frac{r_a}{2} \iint_{(s,t)\times \Omega}  u^{a+\rho}(\tau,x) \dd x \dd \tau  
\\ &\leq 
\iint_{(s,t)\times \Omega} u^{\rho} \left( r_u-\frac{r_a}{2}\langle u\rangle^a-r_b\langle v\rangle^b \right) \dd x\dd t +  \int_{\Omega}  u^{\rho}(s,x) \dd x 
\\ &\leq  \int_{\Omega}  u^{\rho}(s,x) \dd x + C. 
\end{split}
\end{equation}
Applying Young's inequality and then a H\"older inequality to the third term of the left hand side yields
\begin{align}
\label{young holder}
 \iint_{(s,t)\times \Omega} \frac{v}{\langle v \rangle ^{2-\beta}}u^{\frac{\rho}{2}-\frac{\alpha}{2}} \nabla v\cdot \nabla u^{\frac{\rho}{2}+\frac{\alpha}{2}} \dd x\dd \tau \\
\leq \frac12  \iint_{(s,t)\times \Omega} |\nabla u ^{\frac{\rho}{2}+\frac{\alpha}{2}}|^2 \dd x\dd \tau + \frac12 \iint_{(s,t)\times \Omega} \frac{v^2}{\langle v \rangle ^{4-2\beta}}  u^{\rho-\alpha}  |\nabla v|^2 \dd x\dd \tau \notag\\
\leq \frac12 \iint_{(s,t)\times \Omega} |\nabla u ^{\frac{\rho}{2}+\frac{\alpha}{2}}|^2 \dd x\dd \tau + \notag
\\
 + C (1+\max v^{2(\beta-1)_+}) \left(\iint_{(s,t)\times \Omega} u^{\frac{p}{p-2}(\rho-\alpha)} \dd x\dd \tau \right)^{\frac{p-2}{p}} \left(\iint_{(s,t)\times \Omega} |\nabla v|^p \dd x\dd \tau \right)^{\frac{2}{p}}. \notag
\end{align}

 Combining this with \eqref{simpl} gives the result. 
 \end{proof}

%\begin{proof}[Proof of Lemma \ref{power energ u}]
%The proof is the same as Proposition \ref{energy estim for 1st lemma of DG} where we change $\rho +2$ by $\rho$, we take $c_k=0$ and in \eqref{eq: a step of the est} we put $u^\alpha$ inside the gradient to get the left hand side. Moreover we estimate $ \int  (\nabla v^{\beta})^2 u^{\rho-\alpha} \dd x\dd t$ differently to get the right hand side. 
%We apply a H\"older inequality to get, 
%\begin{align*}
%\int  (\nabla v^{\beta})^2 u^{\rho-\alpha} \dd x\dd t\leq \left(\int  (\nabla v^{\beta})^p \dd x\dd t\right)^{\frac{2}{p}}\left(\int    u^{\frac{p}{p-2}(\rho-\alpha)} \dd x\dd t \right)^{\frac{p-2}{p}},
%\end{align*}
%and we obtain the result using that $\nabla v\in L^{p}(\Omega_T)$.
%\end{proof} 
 
  We finally have the following bound on the gradient of $v$ thanks to our assumptions on the parameters.

\begin{proposition}[Integrability of $\nabla v$] \label{3.4 in Tao_Winkler} 
Let $(u,v)$ be a solution of \eqref{eq:the system}.
Suppose \eqref{def cst}--\eqref{assump cst}, in particular that  $2 d < \max\{ 1+a, 2+ \alpha, a+\alpha, 2\alpha+2/m\}.$   Then $$\| \nabla v \|_{L^4 (\Omega_{T})} \le C$$
where $C$ depends only on the initial data, the domain $\Omega$, the time $T$ and the parameters $D$.

Furthermore there exists $\rho>1$ such that
% $1\leq \rho:=  \max\{3\alpha/2+1,2\alpha +a, \alpha+1/2+a/2\}$ we have
\begin{equation} \label{eq: from Tao Winkler}
\int_0^T \int_{\Omega} \left\vert \nabla u^{\frac{\rho+\alpha}{2}} \right\vert^2(s,x) \dd x \dd s \leq C(T).
\end{equation}
%For an increasing  sequence of $\rho_k = \frac{2\rho_{k-1} }{m}+ 3 \alpha$, converging to $ \frac{3m\alpha}{(m-2)_+}$,
%\begin{equation}
%\int_t^{t+\tau} \int_{\Omega} |\nabla u^{\frac{\rho_k+\alpha}{2}}|^2(x,s) \dd x \dd s \leq K. 
%\end{equation}
\end{proposition}
\begin{proof} 
We first prove that $\| \nabla v \|_{L^2 (\Omega_{T})} \le C$. By multiplying the equation of $v$ in \eqref{eq:the system} by $v$ and integrating on $\Omega_T$, one gets
\begin{align}
 \int_\Omega \frac{v^2}2 (T,x) \dd x+ \iint_{\Omega_T} \left(d_v+d_\gamma \left(1+\gamma\frac{v^2}{\langle v \rangle^2} \right)\langle v \rangle ^\gamma \right)|\nabla v|^2 \dd x \dd s \notag\\
\leq \int_\Omega\frac{v^2}2 (0,x) \dd x + \iint_{\Omega_T} v\, (r_v - r_c \langle v\rangle^c-r_d \langle u\rangle^d) \dd x \dd s \notag\\
\leq \int_\Omega \frac{v^2}2 (0,x) \dd x + C.
\end{align}

%Let $\rho\geq \alpha$.
%We define $q= 2+ \frac{4\rho}{m(\rho+\alpha)}$. 
%By the Gagliardo-Nirenberg inequality \cite{nirenberg}, there exists $C>0$ depending on $\rho, m, \alpha$ and $\Omega$ such that 
%\begin{equation}
%\label{GN ineq}
%\|u^{\frac{\rho+\alpha}{2}}(t,.)\|_{L^q(\Omega)} \leq C \|\nabla u^{\frac{\rho+\alpha}{2}} (t,.)\|_{L^2(\Omega)}^{\frac{2}{q}}\| %u^{\frac{\rho+\alpha}{2}}(t,.)\|_{L^{\frac{2\rho}{\rho+\alpha}}}^{\frac{4\rho}{mq(\rho+\alpha)}}+ C\|u^{\frac{\rho+\alpha}{2}}(t,.) \|_{L^{\frac{2\rho}{\rho+\alpha}}(\Omega)}.
%\end{equation}
%Putting \eqref{GN ineq} to the exponent $q$, integrating it in time gives

%\begin{equation}
%\label{GN finale ineq}
%\|u\|_{L^{\rho+ \alpha+ \frac{2\rho}{m}}(\Omega_T)}^{\rho+ \alpha+ \frac{2\rho}{m}} \leq C \left(\sup\limits_{t\in (0,T)} \| u(t,.)\|_{L^{\rho}}^{\frac{2\rho}{m}}+ \sup\limits_{t\in (0,T)} \| u(t,.)\|_{L^{\rho}}^{q\frac{\rho+\alpha}{2}} \right) \left(\|\nabla u^{\frac{\rho+\alpha}{2}} \|_{L^2(\Omega_T)}^2 +1 \right).
%\end{equation}
Then, we recall that from Lemma \ref{power energ u} we have 
\begin{align} \label{estim on u}
 \int_{\Omega} u^{\rho}(T,x) \dd x  + & \iint_{\Omega_T} \vert \nabla u^{\frac{\rho}{2}+\frac{\alpha}{2}} \vert^2(t,x) \dd x \dd t +  \iint_{ \Omega_T}  u^{a+\rho}(t,x) \dd  x \dd t 
    \\ \lesssim&   \int_{\Omega} u^{\rho}(0,x) \dd x+   \iint_{\Omega_T} u^{\rho-\alpha}  |\nabla v|^{2} (t,x) \dd x  \dd t  + \iint_{\Omega_T} u^{\rho} \dd x \dd t \notag.
\end{align} 
For the first term in the right-hand side we apply Young's inequality with a constant $\eta$, to get 
\begin{equation} \label{Youngs with eta}
 \iint_{\Omega_T} u^{\rho-\alpha}  |\nabla v|^{2} (t,x) \dd x \dd t  \leq \eta  \iint_{\Omega_T}  u^{2\rho- 2\alpha}(t,x) \dd x \dd t + \frac{1}{4 \eta}   \iint_{\Omega_T} |\nabla v|^{4} (t,x)\dd x \dd t .
\end{equation}

We recall  $w:=(d_v+d_{\gamma}\langle v \rangle ^\gamma)v$. From Lemma \ref{lem:basic} we have that $w$ satisfies a bound in $L^\infty$.
Furthermore, by integration by part and using the Neumann boundary condition, we know that 
$$\|\Delta w\|_{L^{2}((0,T)\times\Omega)}^2= \sum\limits_{\begin{array}{c}
i=1,\dots,m\\
j=1,\dots,m
\end{array}} \left\|\frac{\partial^2 w}{\partial_{x_i} \partial_{x_j}}\right\|_{L^{2}(\Omega_T)}^2.$$ By the interpolation from Lemma \ref{lem:interpolation} applied for $p=2$, i.e. between $L^\infty$ and $H^2$, we can write
\begin{equation}\label{interpgv}
\iint_{\Omega_T} | \nabla v|^4(t,x) \dd x \dd t  \lesssim \iint_{\Omega_T} | \Delta w|^2(t,x)\dd x \dd t.
\end{equation}
Therefore \eqref{estim on u} together with \eqref{Youngs with eta} gives,
\begin{equation} \label{estim on u 2}
\begin{split}
 \int_{\Omega} u^{\rho}(T,x) \dd x  &+  \iint_{\Omega_T}\vert \nabla u^{\frac{\rho}{2}+\frac{\alpha}{2}} \vert^2(t,x) \dd x \dd t  +  \iint_{ \Omega_T}  u^{a+\rho}(t,x) \dd  x \dd t
    \\ \lesssim& \int_{\Omega} u^{\rho}(0,x) \dd x + \eta \iint_{\Omega_T}  u^{2\rho- 2\alpha}(t,x) \dd x \dd t  + \frac{1}{4 \eta}  \iint_{\Omega_T} |\Delta w|^{2} (t,x)\dd x \dd t \\
    &\qquad + \iint_{\Omega_T} u^{\rho} (t,x)\dd x \dd t.
    \end{split}
\end{equation}

Moreover $w$ satisfies the following equation
\begin{align}
\label{eq in w}
\partial_t w - \left(d_u+d_{\gamma}\langle v \rangle ^{\gamma}\left(1+ \frac{\gamma v^2}{\langle v \rangle ^2}\right)\right)\Delta w 
= \left(d_u+d_{\gamma}\langle v \rangle ^{\gamma}\left(1+ \frac{\gamma v^2}{\langle v \rangle ^2}\right)\right) v(r_v-r_c \langle v\rangle^c-r_d \langle u\rangle^d),
\end{align}
and also $\nabla w \cdot n =0$ on $\partial \Omega$ and $w(0,x)=(d_v+d_{\gamma}\langle v_0 \rangle ^\gamma)v_0$. The estimates we get for $\nabla w$ then are directly transferred to $\nabla v$ through their relation $\nabla v = \left(d_u+d_{\gamma}\langle v \rangle ^{\gamma}\left(1+ \frac{\gamma v^2}{\langle v \rangle ^2}\right)\right)^{-1} \nabla w$.

By multiplying \eqref{eq in w} by $\Delta w$ and integrating in $\Omega_T$ we obtain,
\begin{equation} 
\begin{split}
&\displaystyle \int_{\Omega} \frac{|\nabla w |^2}{2}(T,x) \mathrm{d}x + \displaystyle \iint_{\Omega_T} \left(d_u+d_{\gamma}\langle v \rangle ^{\gamma}+\frac{\gamma d_{\gamma}v^2}{\langle v \rangle ^{2-\gamma}}\right)|\Delta w|^2 \mathrm{d}t \mathrm{d}x\\
&\leq \displaystyle \int_{\Omega} \frac{|\nabla w |^2}{2}(0,x) \mathrm{d}x+  \displaystyle \iint_{\Omega_T} \left(d_u+d_{\gamma}\langle v \rangle ^{\gamma}+\frac{\gamma d_{\gamma}v^2}{\langle v \rangle ^{2-\gamma}}\right)v|r_v-r_c \langle v\rangle^c-r_d \langle u\rangle^d||\Delta w| \mathrm{d}t \mathrm{d}x.
\end{split}
\end{equation}
We use next Young's inequality for the last integral, the fact that $v$ is upper bounded from Lemma \ref{lem:basic} and the inequality $\langle u\rangle^{2d} \leq C(1+u^{2d})$ for some $C$, to get 
\begin{equation} \label{estim on grad w}
\begin{split}
\displaystyle \int_{\Omega} \frac{|\nabla w |^2}{2}(T,x) \mathrm{d}x + \displaystyle \iint_{\Omega_T} \left(d_u+d_{\gamma}\langle v \rangle ^{\gamma}+\frac{\gamma d_{\gamma}v^2}{\langle v \rangle ^{2-\gamma}}\right)|\Delta w|^2 \mathrm{d}t \mathrm{d}x\\
\leq \displaystyle \int_{\Omega} \frac{|\nabla w |^2}{2}(0,x) \mathrm{d}x+  \displaystyle C\iint_{\Omega_T}\left(1+u^{2d}\right)(t,x) \mathrm{d}t \mathrm{d}x,
\end{split}
\end{equation}

%or equivalently the version before the integration in time
%\begin{equation} \label{estim on grad w}
%\begin{split}
%\displaystyle \frac{d}{dt} \int_{\Omega} \frac{|\nabla w |^2}{2}(t,x) \mathrm{d}x + \displaystyle \int_{\Omega} \left(d_u+d_{\gamma}\langle v \rangle ^{\gamma}+\frac{\gamma d_{\gamma}v^2}{\langle v \rangle ^{2-\gamma}}\right)|\Delta w|^2 \mathrm{d}x\\
%\leq \displaystyle C\int_{\Omega}\left(1+u^{2d}\right) \mathrm{d}x
%\end{split}
%\end{equation}
%for all $t \in (0,T)$. 

%Since by assumption $2d\leq \max \{ 1+a, 2+\alpha, a+\alpha \}$ and $u\in L^{\alpha+2}(\Omega_T)$, then we deduce $\Delta w \in L^{2}((0,T)\times\Omega).$

% Then we can write 
%\[
%\int_\Omega | \nabla v|^4(x,t) dx \lesssim \int_\Omega | \Delta w|^2(x,t)dx.
%\]

Gathering therefore the relations \eqref{estim on u}, \eqref{Youngs with eta} and \eqref{estim on grad w} together and neglecting the (time-independent) constants, one has the following inequality
\begin{equation}
\begin{split}
  \int_{\Omega} u^{\rho}(T,x) \dd x + \frac1\eta \int_{\Omega} \frac{|\nabla w |^2}{2}(T,x) \mathrm{d}x &+ \iint_{\Omega_T}\vert \nabla u^{\frac{\rho}{2}+\frac{\alpha}{2}} \vert^2(t,x) \dd x  \dd t
+  \iint_{ \Omega_T}  u^{a+\rho}(t,x) \dd x \dd t
  \\ &+  \frac{1}{\eta} 
 \iint_{\Omega_T} \left(d_u+d_{\gamma}\langle v \rangle^{\gamma} +
 \frac{\gamma d_{\gamma}v^2}{\langle v \rangle ^{2-\gamma}}\right)|\Delta w|^2 \mathrm{d}x \dd t
 \\ &\lesssim
  \eta \iint_{\Omega_T}  u^{2\rho- 2\alpha}(t,x) \dd x\dd t +  \frac1\eta \iint_{\Omega_T}\left(1+u^{2d}\right) \mathrm{d}x \dd t
  \\&+ 
 \iint_{\Omega_T} u^{\rho} \mathrm{d}x \dd t+  \int_{\Omega} u^{\rho}(0,x) \dd x + \frac{1}{\eta} \int_{\Omega} \frac{|\nabla w |^2}{2}(0,x) \mathrm{d}x.
\end{split}
\end{equation}
We continue by considering three different cases regarding the values of $a, \alpha,m$ and defining an appropriate $\rho >1$ in each case that satisfies $\rho > 2\alpha$ and $d < \rho-\alpha$ as well.  We then estimate the above right-hand side by
 \begin{equation}
\begin{split}
 \eta  \iint_{\Omega_T}  u^{2\rho- 2\alpha}(t,x) \dd x \dd t +  \frac1\eta  \iint_{\Omega_T} &\left(1+u^{2d}\right) \mathrm{d}x  \dd t +  \iint_{\Omega_T} u^{\rho}(t,x) \dd x \dd t + \\
 & +  \int_{\Omega} u^{\rho}(0,x) \dd x + \frac{1}{\eta} \int_{\Omega} \frac{|\nabla w |^2}{2}(0,x) \mathrm{d}x\\ 
&\lesssim \left( \eta  \iint_{\Omega_T}  u^{2\rho- 2\alpha}(t,x) \dd x \dd t + C_\eta \right)
\end{split}
\end{equation}
for some positive constant $C_\eta$.

{\bf{Case 1: $\max(1+a,2+\alpha,a+\alpha,4\alpha+4/m)=\max(1+a,2+\alpha)$.}}

Let $\rho:=\alpha + \max(1+a,2+\alpha)/2>1$, so that $2 \rho - 2 \alpha \le \max(1+a,2+\alpha)$. Then by equation \eqref{dual}, we have that $u^{2\rho-2\alpha}$ is in $L^1(\Omega_T)$. Therefore, after an integration in time, taking $\eta=1$ and using the initial regularity,
\begin{equation}\label{magical_energy}
\begin{split}
\int_{\Omega} u^{\rho}(T,x) \dd x + \int_{\Omega} \frac{|\nabla w |^2}{2}(T,x) \mathrm{d}x &+ \int_0^T\int_{\Omega}\vert \nabla u^{\frac{\rho}{2}+\frac{\alpha}{2}} \vert^2(t,x) \dd x \dd t
+ \int_0^T \int_{ \Omega}  u^{a+\rho}(t,x) \dd x \dd t
  \\ &+ 
 \int_0^T\int_{\Omega} \left(d_u+d_{\gamma}\langle v \rangle ^{\gamma} +
 \frac{\gamma d_{\gamma}v^2}{\langle v \rangle ^{2-\gamma}}\right)|\Delta w|^2 \mathrm{d}x \dd t \le
C(T).
\end{split}
\end{equation}

{\bf{Case 2: $\max(1+a,2+\alpha,a+\alpha,4\alpha+4/m)=a+\alpha$.}}
Let $\rho:= 2 \alpha +a>1$, and therefore, $2 \rho - 2 \alpha = \rho +a $, and then, taking $\eta$ small enough, one recovers the same bound as in \eqref{magical_energy}.

{\bf{Case 3: $\max(1+a,2+\alpha,a+\alpha)<4\alpha+4/m$.}}
Let $\rho:=3\alpha +2/m$. Then $\rho>1$ since
\begin{equation*}
\rho=3\alpha +2/m = 4\alpha+4/m - \alpha -2/m > 2+ \alpha - \alpha -2/m = 2-2/m \ge 1.
\end{equation*}

Then in order to bound the right-hand side, we apply Gagliardo-Nirenberg inequality for $1\leq q=\frac{4(\rho-\alpha)}{\rho+\alpha}$: 
\begin{equation}
\begin{split}
 \int_{\Omega} u^{2\rho-2\alpha}(t,x) \mathrm{d} x &=
\|u^{\frac{\rho+\alpha}{2}}(t,.)\|_{L^q(\Omega)}^q \\ 
&\leq 
C \|\nabla u^{\frac{\rho+\alpha}{2}} (t,.)\|_{L^2(\Omega)}^{2}
 \| u^{\frac{\rho+\alpha}{2}}(t,.)\|_{L^{\frac{2}{\rho+\alpha}}}^{q-2}+ 
C\|u^{\frac{\rho+\alpha}{2}}(t,.) \|_{L^{\frac{2}{\rho+\alpha}}(\Omega)}^q.
\end{split}
\end{equation}
for a constant $C$. One then has that
\begin{equation}
\begin{split} 
\|u^{\frac{\rho+\alpha}{2}}(t,.) \|_{L^{\frac{2}{\rho+\alpha}}(\Omega)} &= \int_{\Omega} u\ \mathrm{d}x \leq C\ \text{ for all } t \in (0,T)
\end{split}
\end{equation}
from Lemma \ref{lem:basic}. We integrate the Gagliardo-Nirenberg inequality in time and we choose $\eta$ small enough so that we can absorb the term $\iint_{\Omega_T}\vert \nabla u^{\frac{\rho}{2}+\frac{\alpha}{2}} \vert^2(t,x) \dd x \dd t $ in the left-hand side. Then the inequality in \eqref{magical_energy} holds.

Combining \eqref{magical_energy} and \eqref{interpgv} concludes the proof.

\end{proof}

 We continue by observing that the integrability of $\nabla v$ allows us to obtain integrability for powers of $u$.

 \begin{corollary}
\label{gain integ for grad v implies u}
Let $D$ satisfy \eqref{def cst}--\eqref{assump cst} and let $u, v$ be solutions of \eqref{eq:the system}. If $\nabla v \in L^p(\Omega_{T})$ with $p>2$ then $u^r\in L^1(\Omega_{T})$ for all $r \in \left(0,\max\left\{\frac{p \alpha (m+1)}{(m+2-p)_+}, \frac{p}{2}\alpha + \frac{p-2}{2}a\right\} \right)$ where $m$ is the space dimension. 
\end{corollary}

\begin{proof} First, if $r\in(0,1]$, it is a direct consequence of the mass estimate in Lemma \ref{lem:basic} and Jensen's inequality.

Then, for $r>1$ such that $r<\frac{p}{2}\alpha + \frac{p-2}{2}a$, we use the energy inequality \eqref{power energy est for u} from Lemma \ref{power energ u} with $\rho:=r$. Since $\rho + a > \frac{p}{p-2} (\rho-\alpha)$, one can absorb the right-hand side by the last term of the left-hand side and conclude that $u^r\in L^1(\Omega_{T})$.

Then, for $r>1$ such that $r< \frac{p \alpha (m+1)}{(m+2-p)_+}$, the idea here is to use a Gagliardo-Nirenberg inequality combined with the previous lemma to deduce a sequence of $L^p$ spaces where $u$ belongs. 
Let $\rho\geq \alpha$.
We define $q= 2+ \frac{4\rho}{m(\rho+\alpha)}$. By the Gagliardo-Nirenberg inequality \cite{nirenberg}, there exists $C>0$ depending on $\rho, m, \alpha$ and $\Omega$ such that 
\begin{equation}
\label{GN ineq}
\|u^{\frac{\rho+\alpha}{2}}(t,.)\|_{L^q(\Omega)} \leq C \|\nabla u^{\frac{\rho+\alpha}{2}} (t,.)\|_{L^2(\Omega)}^{\frac{2}{q}}\| u^{\frac{\rho+\alpha}{2}}(t,.)\|_{L^{\frac{2\rho}{\rho+\alpha}}}^{\frac{4\rho}{mq(\rho+\alpha)}}+ C\|u^{\frac{\rho+\alpha}{2}}(t,.) \|_{L^{\frac{2\rho}{\rho+\alpha}}(\Omega)}.
\end{equation}
Putting \eqref{GN ineq} to the exponent $q$, integrating it in time gives

\begin{equation}
\label{GN finale ineq}
\|u\|_{L^{\rho+ \alpha+ \frac{2\rho}{m}}(\Omega_T)}^{\rho+ \alpha+ \frac{2\rho}{m}} \leq C \left(\sup\limits_{t\in (0,T)} \| u(t,.)\|_{L^{\rho}}^{\frac{2\rho}{m}}+ \sup\limits_{t\in (0,T)} \| u(t,.)\|_{L^{\rho}}^{q\frac{\rho+\alpha}{2}} \right) \left(\|\nabla u^{\frac{\rho+\alpha}{2}} \|_{L^2(\Omega_T)}^2 +1 \right).
\end{equation}

Now using Lemma \ref{power energ u} with $s=0$ together with the assumption that $\nabla v\in L^p(\Omega_T)$, we know that 
\begin{equation}
\label{sup ineq}
\sup\limits_{t\in (0,T)} \|u(t,.)\|^{\rho}_{L^{\rho}(\Omega)} \leq C \left( \|u_0\|^{\rho}_{L^{\rho}(\Omega)} + \|u\|_{L^{\frac{p}{p-2}(\rho-\alpha)}(\Omega_T)}^{\rho-\alpha} +1 \right)
\end{equation}
and
\begin{equation}
\label{grad ineq}
\|\nabla u^{\frac{\rho+\alpha}{2}} \|_{L^2(\Omega_T)}^2 \leq C \left( \|u_0\|^{\rho}_{L^{\rho}(\Omega)} + \|u\|_{L^{\frac{p}{p-2}(\rho-\alpha)}(\Omega_T)}^{\rho-\alpha} +1 \right).
\end{equation}

Then reinserting \eqref{sup ineq} and \eqref{grad ineq} in \eqref{GN finale ineq} we get (using the regularity of the initial data)
\begin{equation}
\label{GN and energy ineq}
\|u\|_{L^{\rho+ \alpha+ \frac{2\rho}{m}}(\Omega_T)}^{\rho+ \alpha+ \frac{2\rho}{m}} \leq C \left( 1+ \|u\|_{L^{\frac{p}{p-2}(\rho-\alpha)}(\Omega_T)}^{(\rho-\alpha)(1+q\frac{\rho+\alpha}{2\rho})} \right).
\end{equation}

%Using \eqref{power energy est for u} with $s=0$ and $t=T$ and \eqref{sup ineq}, we deduce that the $L^{\rho+ \alpha+ \frac{2\rho}{m}}(\Omega_T)$-norm of $u$ is bounded by \ari{ the $L^{\rho}(\Omega_T)$-norm (not necessarily)} and the $L^{\frac{p}{p-2}(\rho-\alpha)}(\Omega_T)$-norm of $u$.

For $\eta_0>1$, let us define $\rho_0>\alpha$ by $\eta_0 = \frac{p}{p-2}(\rho_0-\alpha)$, and let us define the sequences $(\eta_k)_k$ and $(\rho_k)_k$ recursively as
\begin{equation}
\eta_{k+1}= \rho_k+ \alpha+ \frac{2\rho_k}{m}, \qquad \eta_{k+1} = \frac{p}{p-2}(\rho_{k+1}-\alpha), \qquad k\ge 0.
\end{equation}
By definition of these sequences, \eqref{GN and energy ineq} gives that if $u\in L^{\eta_k}(\Omega_T)$ then $u\in L^{\eta_{k+1}}(\Omega_T)$.

We can write, for all $k\ge 0$,
\begin{equation}
\eta_{k+1}= 2\left(1+\frac1m\right) \alpha + \frac{p-2}{p}\left(1+\frac2m\right)\eta_k,
\end{equation}
so that studying the limit we get that in the limit $k \longrightarrow \infty$
\begin{eqnarray}
&\eta_{k} \longrightarrow \frac{\alpha p (1+m)}{2+m-p} \qquad & \text{ if } p < m+2 ,\\
&\eta_{k} \longrightarrow \infty \qquad & \text{ if } p \ge m+2,
\end{eqnarray}
and that concludes the proof.
%positive numbers, recursively such that knowing that $u\in L^{\eta_k}(\Omega_T)$ we will be able to deduce, using \eqref{GN and energy ineq}, that $u\in L^{\eta_{k+1}}(\Omega_T)$. 
%The idea is to find $\rho$ such that $\eta_k=\max (\rho,\frac{p}{p-2}(\rho-\alpha))$ and to deduce $\eta_{k+1}=\rho+\alpha+\frac{2\rho}{m}$.
%We consider two cases.  
%In the first one we have $\eta_k=\frac{p}{p-2}(\rho-\alpha)$ if and only if $\rho\geq \frac{p\alpha}{2}$ if and only if $\eta_k\geq \frac{p\alpha}{2}$ so $$\eta_{k+1}=\frac{p-2}{p} \left(1+\frac{2}{m}\right) u_k + 2\alpha \frac{m+1}{m}.$$
%In the second one we have $\eta_k=\rho$ if and only if $\eta_k=\rho\leq \frac{p\alpha}{2}$ so $$\eta_{k+1}= \left(1+\frac{2}{m}\right) \eta_k + \alpha.$$
%So $(\eta_k)_k$ is defined with initial value to be $\eta_0=2+\alpha$ and
%\begin{equation}
%\eta_{k+1}=\left\{\begin{array}{ll}
%\frac{p-2}{p} \left(1+\frac{2}{m}\right) \eta_k + 2\alpha \frac{m+1}{m} & \mbox{ if } \eta_k\geq \frac{p\alpha}{2}\\
%\left(1+\frac{2}{m}\right) \eta_k + \alpha & \mbox{ if } \eta_k\leq \frac{p\alpha}{2}.
%\end{array}\right.
%\end{equation}
%Studying this sequence gives that if $p\geq m+2$ then $$\frac{p-2}{p}\left(1+\frac{2}{m} \right) \geq 1$$ and so $\eta_k$ tends to $\infty$ in both cases. 
%If $p<m+2,$ then we prove that $\eta_k$ tends to $p\alpha\frac{m+1}{m+2-p}$, which concludes the statement of the Lemma. 
\end{proof}
 
 \subsection{Transfer of integrability from $u$ to $\nabla_x v$}

We prove that we can transfer the integrability of $u$ to $\nabla_x v$. To get this result we make use of $W^{1,2}_p$ estimates that we recall
from \cite[Theorem 9.1 Chapter 5 and remark p.351]{LSU68} and \cite[Lemma 2.2]{VTP08}. 

\begin{lemma}[$W^{1,2}_p$ estimates]
\label{W12 estimates}
Let $3<q<\infty$, $w_{0} \in W^{2}_q(\Omega)$ and $w$ be the unique solution in $W^{1,2}_{q}(\Omega_{T})$ of the equation 
\begin{equation}
\left\{\begin{array}{ll}
w_t-a(t,x)\Delta w= f(t,x) & (t,x)\in \Omega_{T},\\
\nabla w \cdot n = 0 & x\in \partial \Omega, t>0, \\
w(t_0,x)=w_{0} & x\in \Omega
\end{array}\right.
\end{equation}
where $a(t,x)$ is a continuous function on $\overline{\Omega_{T}}$ satisfying 
$$0<\lambda \leq a(t,x) \leq \Lambda, \quad \forall (t,x)\in \overline{\Omega_{T}},$$
where $\lambda$ and $\Lambda$ are positive constants. We assume that $f\in L^{q}(\Omega_{T})$. Then there exists a constant $c_q>0$ depending on $q, T, \lambda$ and $\Lambda$ such that 
$$\| w\|_{W^{1,2}_{q}(\Omega_{T})} \leq c_q \left(\|f\|_{L^q(\Omega_{T})}+\|w_{0}\|_{W^{2}_q(\Omega_{T})}\right),$$
where $w_{0}$ satisfies $\nabla w_{0} \cdot n = 0$ on $\partial \Omega$.
\end{lemma}

%\jessica{to apply this lemma we will need in fact that v is continuous, so we need to add some explanations. It may add some restriction on the exponent $d$ depending on the reference we use for the (H\"older) continuity of v}
%\jessica{Would be interesting to get this estimate with another proof than Ladyzenkaja. Is Ladyzenkaja quantitative? See also Lieberman}

\begin{corollary}[Transfer of integrability from $u$ to $\nabla v$]
\label{gain of int of grad v}
Let $D$ satisfy \eqref{def cst}--\eqref{assump cst} and let $u, v$ be solutions of \eqref{eq:the system} with $v$ uniformly continuous on $\Omega_T$. If $u^d\in L^{p}(\Omega_{T})$ for $p\in [2,+\infty)$ then $\nabla v\in L^{2p}(\Omega_{T})$.% In particular, $\nabla v \in L^4 (\Omega_{T})$. 
\end{corollary}

%\begin{lemma}[Gain of integrability of $\nabla v$]
%\label{gain of int of grad v}
%Let $v$ be a solution of \eqref{eq:the system}-\eqref{Neumann bc} where $2d\leq 2+\alpha$. Then $\nabla v\in L^p(\Omega_T)$ for any $p\in (2,+\infty)$. 
%\end{lemma}

\begin{proof}[Proof of Corollary \ref{gain of int of grad v}]

Recall the function $w=(d_v+d_{\gamma}\langle v \rangle ^\gamma)v$ and its equation in \eqref{eq in w}. See that $v$ is uniformly continuous in $\Omega_T$ so continuous in $\overline{\Omega_T}$.
Combining the Lemmas \ref{W12 estimates} and \ref{lem:interpolation},  both $\nabla w$
 and $\nabla v$ are in $L^{2p}(\Omega_T)$ since $$\nabla w= \left(d_u+d_{\gamma}\langle v \rangle ^{\gamma}+\frac{\gamma d_{\gamma}v^2}{\langle v \rangle ^{2-\gamma}}\right)\nabla v.$$
\end{proof}

\subsection{Iteration and conclusion}

We deduce using the previous subsections that $u$ and $\nabla v$ are in any $L^{p}$ space. 

\begin{lemma}
\label{u grad v any Lp space}
Let $u$ and $v$ be solutions of \eqref{eq:the system} where $D$ satisfies \eqref{assump cst}. Then $u$ and $\nabla v$ are in $L^{p}(\Omega_T)$ for all $p \in (1,\infty)$ with a bound only depending on $p$, the initial data, $D$, $T$, $\Omega$ and $m$.
\end{lemma}

\begin{proof}
Let us prove that $v$ is uniformly continuous in $\Omega$ to be able to apply Corollary \ref{gain of int of grad v}.
By Proposition \ref{3.4 in Tao_Winkler} we have that $\nabla v \in L^4(\Omega_T)$, so we can apply Corollary \ref{gain integ for grad v implies u} with $p=4$ and deduce that $u\in L^{r}(\Omega_T)$ for any $r<\max\left(\frac{4\alpha(m+1)}{(m-2)_+}, 2\alpha +a\right) $. Therefore  $u^{d} \in L^{r'}(\Omega_T)$ for any $r'<\left(\frac{4\alpha(m+1)}{d (m-2)_+}, \frac{2\alpha +a}{d}\right)$ where thanks to the condition $d < \max\left(\frac{8\alpha(m+1)}{(m^2-4)_+},\frac{4\alpha +2a}{m+2}\right)$ in \eqref{assump cst}, we can choose $r'>\frac{m+2}{2}$. 
We know that $v$ is solution of the following equation in divergence form 
$$\partial_t v -\nabla \cdot\left( \left(d_v+d_\gamma (1+\gamma \frac{v^2}{\langle v \rangle ^2}\langle v \rangle ^{\gamma}) \right)  \nabla v\right) = v(r_v - r_c \langle v\rangle^c - r_d \langle u\rangle^d),$$
where the diffusion coefficient $\left(d_v+d_\gamma (1+\gamma \frac{v^2}{\langle v \rangle ^2}\langle v \rangle ^{\gamma}) \right)$ is bounded and the right-hand side $v(r_v - r_c \langle v\rangle^c - r_d \langle u\rangle^d)$ in $L^{r'}(\Omega_T)$ for some $r'>\frac{m+2}{2}$.
Applying \cite[Theorem 1.3, p. 43]{D93} or \cite[Theorem 10.1, p. 204]{LSU68}, we deduce that $v$ is H\"older continuous. So in particular uniformly continuous in $\Omega_T$.

Notice now that Corollaries \ref{gain integ for grad v implies u} and \ref{gain of int of grad v} give the following sequence of implications
\begin{equation}
\nabla v \in L^p (\Omega_T) \Rightarrow u^d\in L^{p} (\Omega_T) \Rightarrow \nabla v \in L^{2p} (\Omega_T) \Rightarrow u^d\in L^{2p}  (\Omega_T). 
\end{equation}
Therefore $\nabla v \in L^{4}(\Omega_T)$ yields immediately the result. 
\end{proof}

 \begin{remark} We remark here that one can obtain an $L^p(\Omega)$-estimate for the $\nabla u$ as in equation \eqref{eq: from Tao Winkler} in Lemma \ref{3.4 in Tao_Winkler} uniformly in time following the proof proposed in 
 \cite[Prop. 3.4]{TaoWinkler_stab}. If one further assumes convexity of the domain, following \cite{TaoWinkler_stab} one can make the estimates in Lemma \ref{u grad v any Lp space}  uniform in time under additional assumptions on the parameters. 
 \end{remark}

 \section{$L^\infty$ estimate for $u$ and proof of the main Theorem} \label{sec:final proof} 

We now prove our main theorem. We first prove that the solution $u$ of \eqref{eq:the system} is bounded on $\Omega_T$, and then conclude thanks to classical H\"older and Schauder estimates.

\begin{lemma}
\label{coro u Linf}
For any solution $u:\Omega_T\rightarrow \mathbb{R}$ of \eqref{eq:the system} with $D$ satisfying \eqref{assump cst}, the following estimate holds true,
$$\| u \|_{L^{\infty}(\Omega_T)} \leq C(T),$$
where $C(T)$ only depends on the initial data, $D$, $T$, $\Omega$ and $m$. 
\end{lemma}

\begin{proof} The result can be obtained 
with the Moser iteration technique, see \cite{TaoWinkler_stab}, \cite[Appendix A]{Tao:2011ug}. For the paper to be self-contained we present our proof: it relies on the De-Giorgi technique. The details are presented in the Appendix \ref{app:DG}. 
We prove separately the boundedness on $\left(0,\frac{T}{2} \right)\times \Omega$ and $\left(\frac{T}{2},T \right)\times \Omega$. 

First by a Sobolev embedding (see for example \cite[Theorem p.85]{adams03}) since by Theorem \ref{amann}, $u\in C(\left[0,T\right),W^{1,p_0}(\Omega))$, we deduce 
$$\sup_{t\in \left[0, \frac{T}{2} \right]}\norm{u(t,.)}_{L^{\infty}(\Omega)} \leq C \sup_{t\in \left[0, \frac{T}{2} \right]} ||u(t,.)||_{W^{1,p_0}(\Omega)} \leq C(T).$$ 
Then apply Lemma \ref{lemma L2 Linf} with $\tau_1=0$ and $\tau_2=\frac{T}{2}$ to $\frac{\delta u}{\|u\|_{L^{m+3}(\Omega_T)}+\delta}$ which still satisfies the energy estimate \eqref{energy est for u} for $\rho=m+3$, to deduce that $u$ is bounded in $\left(\frac{T}{2},T\right)\times \Omega$ which ends the proof.
\end{proof}

We now conclude the proof of Theorem \ref{th: main} using the classical H\"older and Schauder estimates.

\begin{proof}
Let us define $A(t,x),F(t,x)$ and $f(t,x), B(t,x)$ and $g(t,x)$ as follows,
$$A(t,x)= d_u+d_{\alpha}(1+\alpha \frac{u^2}{\langle u \rangle ^2}) \langle u \rangle ^{\alpha} + d_{\beta}\langle v \rangle ^{\beta},$$
$$F(t,x)=  \beta d_{\beta} \frac{u v}{\langle v \rangle ^{2-\beta}} \nabla v=G(t,x) \nabla v,$$
$$f(t,x)= u(r_u-r_a\langle u\rangle^a-r_b\langle v\rangle^b).$$
Then $u$ is a solution of 
$$ \partial_t u -\nabla \cdot (A(t,x) \nabla u) + \nabla \cdot  F= f.$$ 

Let us define $B(t,x)$ and $g(t,x)$ as follows,
$$B(t,x)= d_v +d_{\gamma}(1+\gamma \frac{v^2}{\langle v \rangle ^2})\langle v \rangle ^{\gamma},$$
$$g(t,x)= v(r_v-r_c \langle v\rangle^c-r_d \langle u\rangle^d).$$
Then $v$ is a solution of 
$$ \partial_t v -\nabla \cdot (B(t,x) \nabla v) = g.$$ 

From Lemmas \ref{lem:basic}, \ref{u grad v any Lp space} and \ref{lemma L2 Linf},  there are constants $\Lambda>0$ and $C>0$ such that 
$$ \Lambda^{-1}\leq A(t,x) \leq \Lambda \quad \text{and} \quad \Lambda^{-1}\leq B(t,x) \leq \Lambda \quad \forall (t,x) \in \Omega_T,$$ 
$$ \| f\|_{L^{\infty}(\Omega_T)} + \|g \|_{L^{\infty}(\Omega_T)} \leq C,$$ 
$$ \|F \|_{L^{p}(\Omega_T)} \leq C \quad \mbox{for any } p\in (2,\infty).$$

Using for both $u$ and $v$ the classical H\"older regularity theory  \cite[Theorem 10.1, p.204]{LSU68} or \cite[Theorem 1.3, Remark 1.1 p.43]{D93}, there exists constants $\varpi \in (0,1)$ and $C(T)>0$ such that 

$$ \|u\|_{C^{\varpi, \frac{\varpi}{2}}(\overline{\Omega_T})} \leq C(T) \quad \text{and} \quad \|v\|_{C^{\varpi, \frac{\varpi}{2}}(\overline{\Omega_T})} \leq C(T).$$
We define $w=(d_v+d_{\gamma}\langle v \rangle ^\gamma)v$ and $\bar{w}=(d_u+d_\alpha \langle u \rangle ^\alpha+d_\beta \langle v \rangle ^\beta)u$ which solves 
$$\partial_t w = B(t,x) \Delta w +B(t,x) g(t,x),$$
and 
$$ \partial_t \bar{w} = A(t,x) \Delta\bar{w} +A(t,x) f(t,x)+ G(t,x) \Delta w + G(t,x) g(t,x),$$
with homogeneous Neumann boundary conditions. 
We apply the Schauder estimate \cite[Theorem 5.3, pp.320-321]{LSU68} first to $w$ since the coefficients in the equation of $w$ are H\"older continuous, to obtain
$$\|v\|_{C^{\frac{2+\kappa}{2}, 2+\kappa}(\overline{\Omega_T})} \leq C(T),$$
where $\kappa\in \left(0,\nu\right].$
Using the last estimate and the Schauder estimate for the equation of $\bar{w}$, we get 
$$\|u\|_{C^{\frac{2+\mu}{2},2+\mu}(\overline{\Omega_T})} \leq C(T),$$
where $\mu\in \left(0,\nu\right]$. Iterating again, it gives $u,v \in C^{\frac{2+\nu}{2},2+\nu}(\overline{\Omega_T})$ regarding to the initial data. So $\nabla u$ and $\nabla v$ are H\"older continuous on $\overline{\Omega_T}$ so necessarily \eqref{condition on the grad} is satisfied. Consequently, by Theorem \ref{amann} the maximal time of the local solutions is $t_{\mathrm{max}}=\infty$ which ends the proof. 
\end{proof}

\appendix
%\begin{appendix}
 
\section{Boundedness of $u$ with De Giorgi} \label{app:DG}

%In this section we prove that the solution $u$ of \eqref{eq:the system} is bounded. 
In this appendix we get the first lemma of De Giorgi, which is an $L^{\rho}-L^{\infty}$ estimate for $u$. To this end, we define a relevant class of energy estimate which we call the De Giorgi class by analogy with quasi-linear elliptic and parabolic studies \cite{D95,Lieberman}. 

%Before we present our method, we explain how Moser iteration technique can be applied in our system, as was done in \cite{TaoWinkler_stab}, \cite[Appendix A]{Tao:2011ug}. 
 
 %\smallsection{The De Giorgi class}
 %We derive from the equation of $u$ in \eqref{eq:the system} a class of energy estimates, called De Giorgi class, which contains the information to deduce the boundedness of $u$ in the next subsection. 
\begin{proposition}[The De Giorgi class]
\label{energy estim for 1st lemma of DG}
 Considering the system of cross-diffusion equations given by \eqref{eq:the system}, when $\alpha>0$ the following energy estimate holds true for any constant $\rho>2$ and any $k \in \mathbb{R}$,
\begin{align} \label{energy est for u}& \int_{\Omega} (u-k)_{+}^{\rho}(t,x) \dd x  +   \iint_{(s,t)\times\Omega}  \vert \nabla(u-k)_{+}^{\frac{\rho}{2}} \vert^2(\tau,x) \dd x \dd\tau     \\ & \lesssim  \int_{\Omega} (u-k)_{+}^{\rho}(s,x) \dd x + \left( \iint_{(s,t)\times\Omega} (u-k)_{+}^{\rho} (\tau,x) \dd x \dd\tau  \right)^{\frac{\rho-2}{\rho}}  \notag \\ & +  \iint_{(s,t)\times\Omega}  (u-k)_{+}^{\rho-1}(\tau,x) \dd x \dd\tau \notag 
\end{align} 
where the constant only depends on the initial data, $D$, $T$, $\Omega$, $\rho$. 
\end{proposition}

%\begin{remark}
%Since we prove in Lemma \ref{u grad v any Lp space} that $u$ is any $L^p$ space for $p\in (2,\infty)$ %independently of this energy estimate, all terms are well-defined for any $\rho >0$.    
%\end{remark}

%\begin{remark}
%From this class of energy estimate we could also deduce the second lemma of De Giorgi which in an intermediate value lemma applying the same ideas of proof as in \cite{gueDGhalv2}. This implies that we could get the H\"older regularity of $u$ by the De Giorgi method following the same steps as in \cite{gueDGhalv2} by defining the appropriate class of energy estimates for $(u-k)_-$ also. 
%\end{remark}

%\begin{proposition}
%\label{energy estim for 1st lemma of DG}
% Considering the system of cross-diffusion equations given by \eqref{eq:the system}, when $\alpha>0$ the following energy estimate holds true: 
%\begin{align} \label{energy est for u}\frac{1}{\rho+2}&\int_{\Omega} (u-c_k)_{+}^{\rho+2}(x,t) \dd x  +  \int_{\Omega} \int_{s}^t \vert \nabla(u-c_k)_{+}^{1+\frac{\rho}{2}} \vert^2(x,\tau) \dd xd\tau     \\ \lesssim & \frac{1}{\rho+2}  \int_{\Omega} (u-c_k)_{+}^{\rho+2}(x,s) \dd x + c_2 \left( \int_{\Omega} \int_{s}^t(u-c_k)_{+}^{\rho+2} (x,\tau) \dd xd\tau  \right)^{\frac{\rho}{\rho+2}}  \notag \\ & + c_1\int_{\Omega} \int_{s}^t (u-c)_{+}^{\rho+1}(x,\tau) \dd xd\tau \notag 
%\end{align} 
%for some constant $\rho>0$ to be determined later. 
%\end{proposition}

\begin{proof}[Proof of Proposition \ref{energy estim for 1st lemma of DG}]
We multiply the equation for $u$ by $(u-k)^{\rho-1}_+$ and integrate on $(s,t)\times \Omega$,
  \begin{align*}
\iint_{(s,t)\times \Omega} \frac{\partial_t (u-k)_+^{\rho}}{\rho} \dd x\dd t  + \iint_{(s,t)\times \Omega} \beta d_{\beta} \frac{v}{\langle v \rangle ^{2-\beta}}u(u-k)_+^{\frac{\rho}{2}-1} \nabla v\cdot \nabla (u-k)_+^{\frac{\rho}{2}} \dd x\dd t \notag \\ 
+ \iint_{(s,t)\times \Omega}   \Big( (d_u+d_{\alpha}\left(1+\alpha\frac{ u^2}{\langle u \rangle ^2}\right)\langle u \rangle ^{\alpha} + d_{\beta}\langle v \rangle ^{\beta})  \vert \nabla (u-k)_+^{\frac{\rho}{2}}\vert^2  \Big) \dd x\dd t\\
= \iint_{(s,t)\times \Omega}  (u-k)_+^{\rho-1}u (r_u-r_a\langle u\rangle^a-r_b\langle v\rangle^b)\dd x\dd t,
\end{align*}
which using the fact that $-\langle x\rangle^a\le -x^a$ and $x(r_u-r_a x^a)\le C$ for all $x\ge0$, simplifies into 
  \begin{align}
  \label{ener estim simpl}
\int_{\Omega}  \frac{1}{\rho}(u-k)_+^{\rho}(t,x) \dd x\dd t   + 
 \iint_{(s,t)\times \Omega} \beta d_{\beta} \frac{v}{\langle v \rangle ^{2-\beta}}u(u-k)_+^{\frac{\rho}{2}-1} \nabla v\cdot \nabla (u-k)_+^{\frac{\rho}{2}} \dd x\dd t \notag \\ 
 + \iint_{(s,t)\times \Omega}  d_{u}  \vert \nabla (u-k)_+^{\frac{\rho}{2}}\vert^2   \dd x\dd t \notag\\
 \leq \int_{\Omega}  \frac{1}{\rho}(u-k)_+^{\rho}(s,x) \dd x\dd t + \iint_{(s,t)\times \Omega}  (u-k)_+^{\rho-1}\dd x\dd t. 
\end{align}
Applying Young's inequality and then a H\"older inequality to the second term of the left hand side yields
\begin{align}
\label{ener estim young holder}
- \iint_{(s,t)\times \Omega} \beta d_{\beta} \frac{v}{\langle v \rangle ^{2-\beta}}u(u-k)_+^{\frac{\rho}{2}-1} \nabla v\cdot \nabla (u-k)_+^{\frac{\rho}{2}} \dd x\dd t \\
\leq  \iint_{(s,t)\times \Omega} \frac{(\beta d_{\beta})^2}{2d_{u}} \frac{v^2}{\langle v \rangle ^{4-2\beta}}  u^2  |\nabla v|^2 (u-k)_+^{\rho-2}  \dd x\dd t +\iint_{(s,t)\times \Omega} \frac{d_{u}}{2} |\nabla (u-k)_+^{\frac{\rho}{2}}|^2 \dd x\dd t \notag\\
\leq \frac{(\beta d_{\beta})^2}{2d_{u}} (1+\max v^{2\beta}) \left(\iint_{\Omega_T} u^{2\rho}  \dd x\dd t \right)^{\frac{1}{\rho}} \left(\iint_{\Omega_T}|\nabla v|^{2\rho}  \dd x\dd t\right)^{\frac{1}{\rho}} \left(\iint_{(s,t)\times \Omega} (u-k)_+^{\rho} \dd x\dd t\right)^{\frac{\rho-2}{\rho}}  \notag\\
+\iint_{(s,t)\times \Omega} \frac{d_{u}}{2} |\nabla (u-k)_+^{\frac{\rho}{2}}|^2  \dd x\dd t. \notag
\end{align}
 Combining \eqref{ener estim simpl} and \eqref{ener estim young holder} and using Lemmas \ref{lem:basic} and \ref{u grad v any Lp space} gives the result. 
\end{proof}

 \smallsection{First Lemma of De Giorgi}
We get from the De Giorgi class the first lemma of De Giorgi, which is a $L^{\rho}-L^{\infty}$ estimate to conclude that $u$ is bounded.

\begin{lemma}[First Lemma of De Giorgi, $L^\rho-L^{\infty}$ estimate]
\label{lemma L2 Linf}
Let $0\leq \tau_1<\tau_2<T$ and $\rho>m+2$.
There exists a positive universal constant $\delta$ only depending on the initial data, $D$, $T$, $\Omega$, $m$, $\rho$ such that 
for any nonnegative function $u:\Omega_T\rightarrow \mathbb{R}$ satisfying the energy estimate \eqref{energy est for u} the following implication holds true.
If 
$$\displaystyle \iint_{(\tau_1,T)\times\Omega} u^\rho \dd x \dd t \leq \delta,$$
then we have 
$$u\leq \frac{1}{2} \quad \mbox{ in } (\tau_2,T) \times \Omega.$$
\end{lemma}

\begin{proof}

In this proof $C>0$ will denote a constant which will only depend on the initial data, $D$, $T$, $\Omega$, $m$, $\rho$. 
We define for any $t\in (0,T)$, $\Omega_{t,T}=(t,T)\times \Omega$ and
$$U_k=\displaystyle\iint_{\Omega_{t_k,T}} (u-c_k)_+^\rho \mathrm{d}x\mathrm{d}t, $$
where $t_k = \tau_1+(\tau_2-\tau_1)\left(1-\frac{1}{2^k} \right)  $ and $c_k = \frac{1}{2}(1-2^{-k}) $. 
We notice that $t_k$ goes from $\tau_1$ to $\tau_2$ and $c_k$ from $0$ to $\frac{1}{2}$.    
We would like to prove that $U_k$ satisfies the following induction formula $$U_k\leq C^{k} U_{k-2}^{\lambda},$$
where $C>0$ is a universal constant and $\lambda>1$ also.
Defining $V_k=U_{2k}$, the sequence $(V_k)$ satisfy 
$$V_k\leq C^{k} V_{k-1}^{\lambda},$$
and applying \cite[Lemma 3.12]{gueDGhalv2}, we deduce that $V_k=U_{2k}$ tends to $0$ when $V_0=U_0$ is small enough. Moreover we have $U_0 = \displaystyle \iint_{\Omega_{\tau_1,T}} u_{+}^\rho \dd x \dd t$ and $U_{\infty}=  \displaystyle \iint_{\Omega_{\tau_2,T}} \left(u-\frac{1}{2}\right)_{+}^2 \dd x \dd t=0$ and we deduce the result. 

Let us prove the induction formula. 
Let us define the Sobolev exponent  $$\sigma= \left\{ \begin{array}{ll}
\frac{2m}{m-2} &\mbox{ if } m>2, \\
5 & \mbox{ if } m=2,\\
+\infty &\mbox{ if } m=1,
\end{array}\right.$$ 
 in the following Sobolev inequality, for almost every $t\in (t_k,t),$
\begin{equation}
\label{sobolevineq}
\| (u-c_k)_+^{\frac{\rho}{2}}(t,\cdot) \|_{L^\sigma (\Omega)} \leq C(m)\| (u-c_k)_+^{\frac{\rho}{2}}(t,\cdot) \|_{H^{1}(\Omega)},
\end{equation}
where $C(m)$ is a constant which only depends on the dimension $m$, see \cite{adams03}.
Using a H\"older inequality, we have
\begin{align}
\label{lm1 ineq1}
U_k=\displaystyle\iint_{\Omega_{t_k,T}} (u-c_k)_+^\rho \dd x \dd t
\leq \displaystyle\int_{t_k}^{T} \left(\displaystyle\int_{\Omega} (u-c_k)_+^{\frac{\rho \sigma}{2}} (t,\cdot) \mathrm{d}x\right)^{\frac{2}{\sigma}} |\{u(t,\cdot)\geq c_k\}\cap \Omega|^{1-\frac{2}{\sigma}} \mathrm{d}t.
\end{align}
Since $\{u(t,\cdot)\geq c_k \} = \{u(t,\cdot) \geq c_{k-1}+ 2^{-k-1} \}$, we deduce that for $t\in (t_k,T)$ 
\begin{align}
\label{lm1 ineq2}
 |\{u(t,\cdot)\geq c_k\}\cap \Omega|^{1-\frac{2}{\sigma}} &\leq  |\{u(t,\cdot)\geq c_{k-1}+ 2^{-k-1} \}\cap \Omega|^{1-\frac{2}{\sigma}} \nonumber \\
 & \leq \left(2^{\rho (k+1)} \displaystyle\int_{\Omega} (u-c_{k-1})_+^{\rho}(t,\cdot) \dd x \right)^{1-\frac{2}{\sigma}} \nonumber \\
 &\leq C^k \left( \sup\limits_{t\in (t_k,T)} \displaystyle\int_{\Omega} (u-c_{k-1})_+^{\rho}(t,\cdot)\dd x \right)^{1-\frac{2}{\sigma}} .
\end{align}
We can use the first part of the inequality defining the De Giorgi class in Lemma \ref{energy estim for 1st lemma of DG} with $s$ integrated in $(t_{k-1},t_{k})$ to bound the supremum and obtain in \eqref{lm1 ineq2}, 
 \begin{align}
\label{lm1 ineq3}
& |\{u(t,\cdot)\geq c_k\}\cap \Omega|^{1-\frac{2}{\sigma}} \nonumber \\
 \leq & C^k \left(\displaystyle\iint_{\Omega_{t_{k-1},T}}(u-c_{k-1})_{+}^{\rho} \dd x \dd t +  \left(\displaystyle\iint_{\Omega_{t_{k-1},T}}(u-c_{k-1})_{+}^\rho  \dd x \dd t\right)^{\frac{\rho-2}{\rho}} + \iint_{\Omega_{t_{k-1},T}}(u-c_{k-1})_{+}^{\rho-1}  \dd x \dd t \right)^{1-\frac{2}{\sigma}} \nonumber \\
 \leq &  C^k \left(\displaystyle\iint_{\Omega_{t_{k-1},T}}(u-c_{k-1})_{+}^{\rho}  \dd x \dd t +  \displaystyle\iint_{\Omega_{t_{k-1},T}}\mathbbm{1}_{\{u\geq c_{k-1} \}}  \dd x \dd t \right)^{1-\frac{2}{\sigma}} \nonumber\\
& + C^k \left(\displaystyle\iint_{\Omega_{t_{k-1},T}}(u-c_{k-1})_{+}^{\rho}  \dd x \dd t +  \displaystyle\iint_{\Omega_{t_{k-1},T}}\mathbbm{1}_{\{u\geq c_{k-1} \}}  \dd x \dd t \right)^{\frac{\rho-2}{\rho}(1-\frac{2}{\sigma})}
\end{align}
 %&  \leq  C^k \left(U_{k-1} + U_{k-1}^{\frac{1}{2}} \right)^{1-\frac{2}{p}},
where we used that $(u-c_{k-1})_+^{\rho-1}\leq  (u-c_{k-1})_+^{\rho}+ \mathbbm{1}_{\{u\geq c_{k-1} \}}$ to get the last bound. 
Since we have 
\begin{align*}
  \displaystyle \iint_{\Omega_{t_{k-1},T}} \mathbbm{1}_{\{u\geq c_{k-1} \}} \dd x \dd t & = |\{u\geq c_{k-1} \}\cap \Omega |\\
 & \leq |\{u(t,\cdot)\geq c_{k-2}+ 2^{-k} \}\cap \Omega| \\
& \leq 2^{\rho k} \displaystyle\int_{\Omega} (u-c_{k-2})_+^\rho \\
& \leq 2^{\rho k} U_{k-2},
\end{align*}
we deduce using \eqref{lm1 ineq3}, 
\begin{align}
\label{lm1 ineq4}
 |\{u(t,\cdot)\geq c_k\}\cap \Omega|^{1-\frac{2}{\sigma}} & \leq C^{k} \left( \left( U_{k-1} + U_{k-2} \right)^{1-\frac{2}{\sigma}}+\left( U_{k-1} + U_{k-2} \right)^{\frac{\rho-2}{\rho}(1-\frac{2}{\sigma})} \right)\nonumber\\
 & \leq  C^{k} \left(U_{k-2}^{1-\frac{2}{\sigma}} + U_{k-2}^{\frac{\rho-2}{\rho}(1-\frac{2}{\sigma})}\right)
\end{align}
We notice that the last bound is independent of the variable $t$ so it remains to bound $\displaystyle\int_{t_k}^{T} \left(\displaystyle\int_{\Omega} (u-c_k)_+^{\frac{\rho \sigma}{2}} \mathrm{d}x\right)^{\frac{2}{\sigma}} \mathrm{d}t$ in \eqref{lm1 ineq1}.
Using the Sobolev inequality \eqref{sobolevineq} and the second part of the inequality defining the De Giorgi class in Lemma \ref{energy estim for 1st lemma of DG} with $s$ integrated in $(t_{k-1},t_k)$, we deduce
 \begin{align}
 \label{lm1 ineq5}
& \displaystyle\int_{t_k}^{t} \left(\displaystyle\int_{\Omega} (u-c_k)_+^{\frac{\rho\sigma}{2}} \dd x \right)^{\frac{2}{\sigma}} \dd t \nonumber \\
 \leq &  C\left(\displaystyle\iint_{\Omega_{t_k,T}} (u-c_k)_+^\rho \dd x \dd t +\displaystyle \iint_{\Omega_{t_k,T}} |\nabla_x (u-c_k)_+^{\frac{\rho}{2}}|^2 \dd x \dd t \right) \nonumber \\
 \leq  & C \left(\displaystyle\iint_{\Omega_{t_{k-1},T}}(u-c_{k-1})_{+}^{\rho} \dd x \dd t +  \left(\displaystyle\iint_{\Omega_{t_{k-1},T}}(u-c_{k-1})_{+}^{\rho} \dd x \dd t \right)^{\frac{\rho-2}{\rho}}+\iint_{\Omega_{t_{k-1},T}}(u-c_{k-1})_{+}^{\rho-1} \dd x \dd t \right)\nonumber \\
\leq & C \left(\displaystyle\iint_{\Omega_{t_{k-1},T}}(u-c_{k-1})_{+}^{\rho} \dd x \dd t+  \displaystyle\iint_{\Omega_{t_{k-1},T}}\mathbbm{1}_{\{u\geq c_{k-1} \}} \dd x \dd t \right)\nonumber\\
& + C \left(\displaystyle\iint_{\Omega_{t_{k-1},T}}(u-c_{k-1})_{+}^{\rho} \dd x \dd t +  \displaystyle\iint_{\Omega_{t_{k-1},T}}\mathbbm{1}_{\{u\geq c_{k-1} \}} \dd x \dd t \right)^{\frac{\rho-2}{\rho}} \nonumber\\
 \leq &  C^k \left(U_{k-2}+ U_{k-2}^{\frac{\rho-2}{\rho}}\right).
 \end{align}
By definition $U_k$ is non-increasing so assuming that $U_0<1$, we have $U_k<1$ for every $k\geq 0$. 
Combining \eqref{lm1 ineq4} and \eqref{lm1 ineq3} and assuming $U_0<1$, we deduce that $U_k$ satisfies the formula
$$U_k\leq C^k \left(U_{k-2}+ U_{k-2}^{\frac{\rho-2}{\rho}}\right) \left(U_{k-2}^{1-\frac{2}{\sigma}} + U_{k-2}^{\frac{\rho-2}{\rho}(1-\frac{2}{\sigma})}\right) \leq C^{k} U_{k-2}^{\lambda},$$
with $\lambda=\frac{\rho-2}{\rho}\left(2-\frac{2}{\sigma}\right)>1$ for $\rho>m+2$ which ends the proof using \cite[Lemma 3.12]{gueDGhalv2} choosing $\delta<\min\left(C^{-\frac{\lambda^2}{(\lambda-1)^2}},1\right)$.
\end{proof}

%\begin{lemma}
%\label{coro u Linf}
%For any solution $u:\Omega_T\rightarrow \mathbb{R}$ of \eqref{eq:the system} with $D$ satisfying \eqref{assump cst}, the following estimate holds true,
%$$\| u \|_{L^{\infty}(\Omega_T)} \leq C(T),$$
%where $C(T)$ only depends on the initial data, $D$, $T$, $\Omega$, $m$. 
%\end{lemma}

%\end{appendix}

\subsection*{Acknowledgements.}  
The authors thank Cl\'ement Mouhot for fruitful discussions about the subject. JG acknowledges the support of partial funding by the ERC grant MAFRAN 2017-2022, AM the support by the EPSRC grant EP/L016516/1 for the University of Cambridge, CCA and the Huawei fellowship at IHES. Part of this work was done while AT enjoyed the warm atmosphere of DPMMS, Univ. Cambridge.

\bibliographystyle{alpha}
\bibliography{bibliography}

\begin{thebibliography}{DLMT15}

\bibitem[AF03]{adams03}
Robert~A Adams and John~JF Fournier.
\newblock {\em Sobolev spaces}, volume 140.
\newblock Elsevier, 2003.

\bibitem[Ama89]{Am89}
Herbert Amann.
\newblock Dynamic theory of quasilinear parabolic systems. {III}. {G}lobal
  existence.
\newblock {\em Math. Z.}, 202(2):219--250, 1989.

\bibitem[Ama90]{Am90}
Herbert Amann.
\newblock Dynamic theory of quasilinear parabolic equations. {II}.
  {R}eaction-diffusion systems.
\newblock {\em Differential Integral Equations}, 3(1):13--75, 1990.

\bibitem[Ama93]{Am93}
Herbert Amann.
\newblock Nonhomogeneous linear and quasilinear elliptic and parabolic boundary
  value problems.
\newblock In {\em Function spaces, differential operators and nonlinear
  analysis ({F}riedrichroda, 1992)}, volume 133 of {\em Teubner-Texte Math.},
  pages 9--126. Teubner, Stuttgart, 1993.

\bibitem[BL12]{BL12}
J{\"o}ran Bergh and J{\"o}rgen L{\"o}fstr{\"o}m.
\newblock {\em Interpolation spaces: an introduction}, volume 223.
\newblock Springer Science \& Business Media, 2012.

\bibitem[BM19]{BM19}
Ha\"{\i}m Brezis and Petru Mironescu.
\newblock Where {S}obolev interacts with {G}agliardo-{N}irenberg.
\newblock {\em J. Funct. Anal.}, 277(8):2839--2864, 2019.

\bibitem[CLY04]{ChoiLuiYamada04}
Y.~S. Choi, Roger Lui, and Yoshio Yamada.
\newblock Existence of global solutions for the
  {S}higesada-{K}awasaki-{T}eramoto model with strongly coupled
  cross-diffusion.
\newblock {\em Discrete Contin. Dyn. Syst.}, 10(3):719--730, 2004.

\bibitem[DiB93]{D93}
Emmanuele DiBenedetto.
\newblock {\em Degenerate parabolic equations}.
\newblock Universitext. Springer-Verlag, New York, 1993.

\bibitem[DiB95]{D95}
Emmanuele DiBenedetto.
\newblock {\em Partial differential equations}.
\newblock Birkh\"{a}user Boston, Inc., Boston, MA, 1995.

\bibitem[DLMT15]{DLMT}
L.~Desvillettes, T.~Lepoutre, A.~Moussa, and A.~Trescases.
\newblock On the entropic structure of reaction-cross diffusion systems.
\newblock {\em Communications in Partial Differential Equations},
  40(9):1705--1747, 2015.

\bibitem[DT15]{DT15}
L.~Desvillettes and A.~Trescases.
\newblock New results for triangular reaction cross diffusion system.
\newblock {\em J. Math. Anal. Appl.}, 430(1):32--59, 2015.

\bibitem[GA73]{GA73}
ME~Gilpin and FJ~Ayala.
\newblock Global models of growth and competition.
\newblock {\em Proc Natl Acad Sci U S A}, page 70(12):3590‐3593, 1973.

\bibitem[GJ72]{GJ72}
ME~Gilpin and KE~Justice.
\newblock Animal dispersion in relation to social behavior.
\newblock {\em Nature}, pages 236(8):273--303, 1972.

\bibitem[Gue20]{gueDGhalv2}
Jessica Guerand.
\newblock {Quantitative regularity for parabolic De Giorgi classes}.
\newblock hal-02069086, January 2020.

\bibitem[HNP15]{HNP2}
L.~Hoang, T.~Nguyen, and T.~Phan.
\newblock Gradient estimates and global existence of smooth solutions to a
  cross-diffusion system.
\newblock {\em SIAM Journal on Mathematical Analysis}, 47(3):2122--2177, 2015.

\bibitem[Le21]{Le21}
Dung Le.
\newblock On the global existence of a class of strongly coupled parabolic
  systems.
\newblock arXiv:1910.07260, November 2021.

\bibitem[Lie96]{Lieberman}
Gary~M. Lieberman.
\newblock {\em Second order parabolic differential equations}.
\newblock World Scientific Publishing Co., Inc., River Edge, NJ, 1996.

\bibitem[LNN03]{LeNguyenNguyen}
Dung Le, Linh~Viet Nguyen, and Toan~Trong Nguyen.
\newblock Shigesada-{K}awasaki-{T}eramoto model on higher dimensional domains.
\newblock {\em Electron. J. Differential Equations}, pages No. 72, 12, 2003.

\bibitem[LNW98]{LouNiWu}
Yuan Lou, Wei-Ming Ni, and Yaping Wu.
\newblock On the global existence of a cross-diffusion system.
\newblock {\em Discrete Contin. Dynam. Systems}, 4(2):193--203, 1998.

\bibitem[LSU68]{LSU68}
O.~A. Ladyzhenskaya, V.~A. Solonnikov, and N.~N. Ural'tseva.
\newblock {\em Linear and quasilinear equations of parabolic type}.
\newblock Translated from the Russian by S. Smith. Translations of Mathematical
  Monographs, Vol. 23. American Mathematical Society, Providence, R.I., 1968.

\bibitem[Nir59]{nirenberg}
L~Nirenberg.
\newblock On elliptic partial differential equations.
\newblock {\em Annali della Scuola Normale Superiore di Pisa-Classe di
  Scienze}, 13(2):115--162, 1959.

\bibitem[PT90]{PozioTesei}
M.~A. Pozio and A.~Tesei.
\newblock Global existence of solutions for a strongly coupled quasilinear
  parabolic system.
\newblock {\em Nonlinear Anal.}, 14(8):657--689, 1990.

\bibitem[SKT79]{SKT}
Nanako Shigesada, Kohkichi Kawasaki, and Ei~Teramoto.
\newblock Spatial segregation of interacting species.
\newblock {\em J. Theoret. Biol.}, 79(1):83--99, 1979.

\bibitem[Tre16]{AT16}
A.~Trescases.
\newblock On triangular reaction cross-diffusion systems with possible
  self-diffusion.
\newblock {\em Bull. Sci. Math.}, 140(7):796--829, 2016.

\bibitem[TW11]{Tao:2011ug}
Youshan Tao and Michael Winkler.
\newblock Boundedness in a quasilinear parabolic-parabolic keller-segel system
  with subcritical sensitivity.
\newblock {\em Journal of Differential Equations}, 252, 06 2011.

\bibitem[TW19]{TaoWinkler_stab}
Youshan Tao and Michael Winkler.
\newblock Boundedness and stabilization in a population model with
  cross-diffusion for one species.
\newblock {\em Proc. London Math. Soc.}, 119(3):1598--1632, 2019.

\bibitem[VT08]{VTP08}
Phan Van~Tuoc.
\newblock On global existence of solutions to a cross-diffusion system.
\newblock {\em Journal of Mathematical Analysis and Applications},
  343(2):826--834, 2008.

\bibitem[Wan05]{Wang2005}
Yi~Wang.
\newblock The global existence of solutions for a cross-diffusion system.
\newblock {\em Acta Math. Appl. Sin. Engl. Ser.}, 21(3):519--528, 2005.

\bibitem[Yag93]{Yagi}
Atsushi Yagi.
\newblock Global solution to some quasilinear parabolic system in population
  dynamics.
\newblock {\em Nonlinear Anal.}, 21(8):603--630, 1993.

\bibitem[Yam95]{Yamada}
Yoshio Yamada.
\newblock Global solutions for quasilinear parabolic systems with
  cross-diffusion effects.
\newblock {\em Nonlinear Anal.}, 24(9):1395--1412, 1995.

\end{thebibliography}

\end{document}